\documentclass{amsart}

\usepackage{amsmath, amssymb, amsthm}

\newtheorem{thm}{Theorem}[section]
\newtheorem{cor}[thm]{Corollary}
\newtheorem{lem}[thm]{Lemma}
\newtheorem{pro}[thm]{Proposition}
\theoremstyle{definition}
\newtheorem{dfn}[thm]{Definition}
\newtheorem{exa}[thm]{Example}
\newtheorem{rmk}[thm]{Remark}

\def\C{{\mathbb C}}
\def\id{{\textrm{id}}}
\def\O{{\textrm{O}}}
\def\R{{\mathbb R}}
\def\Sp{{\textrm{Sp}}}
\def\U{{\textrm{U}}}

\title[Epsilon-symplectic rigidity]{Epsilon-non-squeezing and $C^0$-rigidity of epsilon-symplectic embeddings}

\author{Stefan M\"uller}

\address{Stanford University, Department of Mathematics, 450 Serra Mall, Building 380, Stanford, CA 94305}
\email{stefanmueller@stanford.edu}
\address{Georgia Southern University, Department of Mathematical Sciences, 65 Georgia Ave.\ Room 3008, P.O.\ Box 8093, Statesboro, GA 30460}



\begin{document}
\thispagestyle{plain}

\begin{abstract}
An embedding $\varphi \colon (M_1, \omega_1) \to (M_2, \omega_2)$ (of symplectic manifolds of the same dimension) is called $\epsilon$-symplectic if the difference $\varphi^* \omega_2 - \omega_1$ is $\epsilon$-small with respect to a fixed Riemannian metric on $M_1$.
We prove that if a sequence of $\epsilon$-symplectic embeddings converges uniformly (on compact subsets) to another embedding, then the limit is $E$-symplectic, where the number $E$ depends only on $\epsilon$ and $E (\epsilon) \to 0$ as $\epsilon \to 0$.
This generalizes $C^0$-rigidity of symplectic embeddings, and answers a question in topological quantum computing by Michael Freedman.

As in the symplectic case, this rigidity theorem can be deduced from the existence and properties of symplectic capacities.
An $\epsilon$-symplectic embedding preserves capacity up to an $\epsilon$-small error, and linear $\epsilon$-symplectic maps can be characterized by the property that they preserve the symplectic spectrum of ellipsoids (centered at the origin) up to  an error that is $\epsilon$-small.

We sketch an alternative proof using the shape invariant, which gives rise to an analogous characterization and rigidity theorem for $\epsilon$-contact embeddings.
\end{abstract}

\maketitle

\section{Introduction and main results} \label{sec:intro}

In this paper, we consider smooth manifolds $M$ equipped with a symplectic structure $\omega$ and a Riemannian metric $g$.
We do not necessarily assume that the metric is compatible with the symplectic structure, or that the induced volume forms coincide (up to a constant multiple), though some of the estimates in this article are more explicit in those cases.

\begin{dfn}[Epsilon-symplectic and epsilon-anti-symplectic]
Let $(M_1, \omega_1)$ and $(M_2, \omega_2)$ be two symplectic manifolds of the same dimension, $g$ be a Riemannian metric on $M_1$, and $\epsilon \ge 0$.
An embedding $\varphi \colon M_1 \to M_2$ is called $\epsilon$-symplectic if $\| \varphi^* \omega_2 - \omega_1 \|_2 \le \epsilon$, and $\epsilon$-anti-symplectic if $\| \varphi^* \omega_2 + \omega_1 \|_2 \le \epsilon$. \qed
\end{dfn}

See section~\ref{sec:norms} for the definition of the norm $\| \cdot \|_2$ and a number of general related results.
A goal of this paper is to prove the following rigidity theorem.

\begin{thm} \label{thm:eps-symp-rig}
Let $(M_1, \omega_1)$ and $(M_2, \omega_2)$ be two symplectic manifolds of the same dimension, and $g$ be a Riemannian metric on $M_1$.
Then there are constants $\delta = \delta (\omega_1, g) > 0$ and $E = E (\omega_1, g, \epsilon) \ge 0$ with $E \to 0^+$ as $\epsilon \to 0^+$ so that, if $\epsilon < \delta$ and $\varphi_k \colon M_1 \to M_2$ is a sequence of $\epsilon$-symplectic embeddings that converges uniformly (on compact subsets) to an embedding $\varphi \colon M_1 \to M_2$, then $\varphi$ is $E$-symplectic.
\end{thm}

In case $(M_1, \omega_1)$ and $(M_2, \omega_2)$ are both subsets of $(\R^{2 n}, \omega_0)$ with its standard symplectic structure and standard flat metric, an explicit lower bound for $\delta$ and explicit upper bound for $E$ can be derived from the proof given below.

\begin{cor} \label{cor:eps-symp-rig}
Let $(M_1, \omega_1)$ and $(M_2, \omega_2)$ be two symplectic manifolds of the same dimension, and $\epsilon_k \ge 0$ be a sequence of non-negative numbers so that $\epsilon_k \to 0^+$ as $k \to \infty$.
Suppose that $\varphi_k \colon M_1 \to M_2$ is a sequence of embeddings that converges uniformly (on compact subsets) to another embedding $\varphi \colon M_1 \to M_2$, and that each $\varphi_k$ is $\epsilon_k$-symplectic.
Then the limit $\varphi$ is a symplectic embedding.
\end{cor}

The choice of Riemannian metric on $M_1$ is not relevant for the corollary.
See Remark~\ref{rmk:arbitrary-metric}.
Analogous to the symplectic case, we show that an embedding is $\epsilon$-symplectic or $\epsilon$-anti-symplectic if and only if it preserves the capacity of ellipsoids up to an $\epsilon$-small error.
Most of the paper is devoted to establishing its linear version on $\R^{2 n}$ with its standard symplectic structure and Riemannian metric.

\begin{pro} \label{pro:eps-non-squeez-expand}
Let $0 \le \epsilon < 1 / \sqrt{2}$, and $\epsilon' = \sqrt{2} \, \epsilon$.
Then an $\epsilon$-symplectic linear map $\Phi \colon \R^{2 n} \to \R^{2 n}$ is $\epsilon'$-non-squeezing and $\epsilon'$-non-expanding.
\end{pro}

The constant $1 / \sqrt{2}$ is not optimal; see section~\ref{sec:norms} for details.
By Remark~\ref{rmk:non-squeez-fail} below, there is no form of non-squeezing for $\epsilon$-symplectic embeddings with $\epsilon \ge 1$.

\begin{thm} \label{thm:eps-non-squeez-expand}
Suppose a linear map $\Phi \colon \R^{2 n} \to \R^{2 n}$ has the linear $\epsilon$-non-squeezing and linear $\epsilon$-non-expanding property.
Then for $\epsilon \ge 0$ sufficiently small, $\Phi$ is either $\epsilon'$-symplectic or $\epsilon'$-anti-symplectic, where $\epsilon' = K (\epsilon) \to 0^+$ as $\epsilon \to 0^+$.
\end{thm}

See sections~\ref{sec:eps-symp-r2n} and \ref{sec:lin-eps-non-squeez} for details.
Symplectic capacities are discussed in section~\ref{sec:rigidity}.
A geometric expression of $\epsilon$-symplectic rigidity is the following generalization of Gromov's non-squeezing theorem.
Consider again $\R^{2 n}$ with its standard symplectic structure $\omega_0$.
Denote by $B_r^{2 n} \subset \R^{2 n}$ the (closed) ball of radius $r > 0$ (centered at the origin), and by $Z_R^{2 n} = B_R^2 \times \R^{2 n - 2}$ the (symplectic) cylinder of radius $R > 0$.

\begin{pro}[Epsilon-non-squeezing] \label{pro:eps-non-squeez}
If there is an $\epsilon$-symplectic embedding of $B_r^{2 n}$ into $Z_R^{2 n}$, with $0 \le \epsilon < 1 / \sqrt{2}$, then $r \le (1 - \sqrt{2} \, \epsilon)^{- \sqrt{2 n}} R$.
\end{pro}

Recall that Gromov's non-squeezing theorem (the case $\epsilon = 0$) can be considered as a geometric expression of the uncertainty principle \cite[page 458]{mcduff:ist17}.
Given a point $(x_1, y_1, \ldots, x_n, y_n)$ in $\R^{2 n} = T^* \R^n$, think of $x_j$ as the $j$-th position coordinate and $y_j$ as the $j$-th momentum coordinate of some Hamiltonian system.
If the state of the system is measured to lie somewhere in a subset $U \subset \R^{2 n}$ that is (or contains) a ball of radius $r$, then the range of uncertainty (to the extend of our knowledge) of the values of the conjugate pair $(x_j, y_j)$ is the area $\pi r^2$.
Proposition~\ref{pro:eps-non-squeez} then means that if the system is transformed by an $\epsilon$-symplectic diffeomorphism, this range of uncertainty can be decreased by a factor of at most $(1 - \sqrt{2} \, \epsilon)^{2 n}$.

The results of this paper are of interest in symplectic integrator methods and topological quantum computing, where computations can be performed up to any prescribed level of accuracy only.
The question by Michael Freedman \cite{freedman:email17} was the starting point of this paper.
Corollary~\ref{cor:eps-symp-rig} is also relevant in $C^0$-symplectic topology.

For most of the paper, we assume that $M$ is compact or a relatively compact subset $U$ of $\R^{2 n}$.
In the latter case, we also assume that there exists a Riemannian metric $\overline{g}$ defined on a neighborhood of $\overline{U}$ such that $\overline{g}_{|_U} = g$.
In particular, all of the supremums considered below are in fact maximums, and in particular, are finite.
See the (first paragraph of the) proof of Theorem~\ref{thm:eps-symp-rig} in section~\ref{sec:rigidity} for the case of non-compact manifolds.
An alternate argument using the shape invariant, and $\epsilon$-contact embeddings, are discussed in the final section~\ref{sec:shape-contact}.

\section{Norms of vector fields and differential forms} \label{sec:norms}
The Riemannian metric $g$ induces a norm on each tangent space $T_x M$ given by $\| v \|_2 = \sqrt{g (v, v)}$ for $v \in T_x M$.
The norm of a vector field $X$ on $M$ is then defined by $\| X \|_2 = \sup_{x \in M} \| X (x) \|_2$.
Let $\star$ denote the Hodge star of the metric $g$.
Then for a $k$-covector $v^* \in \Lambda^k (T_x M)$, let $\| v^* \|_2 = \sqrt{\star (v^* \wedge \star v^*)}$, and for a differential form $\beta$, define $\| \beta \|_2 = \sup_{x \in M} \| \beta (x) \|_2$.
We can also define the comass norms
\[ \| v^* \|_C = \sup \{ v^* (v_1, \ldots, v_k) \mid \| v_1 \|_2 = \cdots = \| v_k \|_2 = 1 \} \]
and $\| \beta \|_C = \sup_{x \in M} \| \beta (x) \|_C$.
The norms $\| \cdot \|_2$ and $\| \cdot \|_C$ are in fact equivalent.
We sketch a proof to the degree necessary for our purposes.
See \cite[Chapter~1]{federer:gmt69} for instance for details.

\begin{lem} \label{lem:norms-equiv}
Let $m = \dim M$.
Then $\| \cdot \|_C \le \| \cdot \|_2 \le \sqrt{m \choose k} \| \cdot \|_C$ for any $k$-covector and any $k$-form, and $\| v^* \|_C = \| v^* \|_2$ if and only if the $k$-covector $v^*$ is simple.
\end{lem}

\begin{proof}[Sketch of proof]
Note that it suffices to prove the lemma for covectors.
The natural isomorphism $\gamma \colon T_x M \to T_x^* M$ given by $\gamma (v) = g (v, \cdot)$ extends to an isomorphism $\gamma \colon \Lambda_k (T_x M) \to \Lambda^k (T_x M)$ for each $k$, and thus the metric $g$ extends to a metric on the space of $k$-vectors given by $(v, w) \mapsto \gamma (v) (w)$.
The induced norm $\| \cdot \|_2$ on $k$-vectors is dual to the norm $\| \cdot \|_2$ for $k$-covectors.
In particular,
\[ \| v^* \|_2 = \sup \{ v^* (v) \mid v \in \Lambda_k (T_x M) \mbox{ with } \| v \|_2 = 1 \}, \]
whereas
\[ \| v^* \|_C = \sup \{ v^* (v_1 \wedge \ldots \wedge v_k) \mid \| v_1 \wedge \ldots \wedge v_k \|_2 = 1 \}, \]
i.e., the latter supremum is over all simple unit $k$-vectors only, where a $k$-(co-)vector is called simple if it is the (alternating) product of $1$-(co-)vectors.
That proves the first inequality, and the claim that the two norms coincide on simple $k$-covectors.

To prove the second inequality, choose an orthonormal basis $e_1, \ldots, e_m$ of $T_x M$, with dual orthonormal basis $\alpha_1, \ldots, \alpha_m$.
Let $v^*$ be a $k$-covector, and write
\begin{align} \label{eqn:covector}
v^* = \sum_\sigma f_\sigma \, \alpha_{\sigma (1)} \wedge \ldots \wedge \alpha_{\sigma (k)},
\end{align}
where the sum is over all strictly increasing functions $\sigma \colon \{ 1, \ldots, k \} \to \{ 1, \ldots, m \}$.
Let $d = {m \choose k}$, and choose some order on the set (with $d$ elements) of such functions.
Denote by $f_{v^*}$ the vector $( f_{\sigma_1}, \ldots, f_{\sigma_d} )$ in $\R^d$, equipped with the standard metric $g_0 = \langle \cdot, \cdot \rangle$.
Then it follows immediately from the definitions that $\| v^* \|_2 = \| f_{v^*} \|_2$.

Let $v_j = \sum_{i = 1}^m \lambda_{i j} e_i$, $1 \le j \le k$, be unit vectors in $T_x M$, and consider the ($m \times k$)-matrix $\Lambda = [ \lambda_{i j} ]_{1 \le i \le m, 1 \le j \le k}$.
Then
\[ v^* (v_1, \ldots, v_k) = \left\langle ( f_{\sigma_1}, \ldots, f_{\sigma_d} ), ( \det (M_{\sigma_1}), \ldots, \det (M_{\sigma_d}) ) \right\rangle, \]
where $M_\sigma$ is the $(k \times k)$-minor obtained from $\Lambda$ by deleting all but the rows in the image of the function $\sigma$.
(Geometrically, the minor $M_\sigma$ represents the linear transformation $\Lambda \circ \Pi_\sigma \colon \R^k \to \R^k$, where $\Pi_\sigma \colon \R^m \to \R^k$ denotes the projection to the components that belong to the image of $\sigma$.
In particular, the absolute value of its determinant can be interpreted as the hyper-volume of the image of a $k$-dimensional face of the unit cube.)
Let $N \le d$ be the number of non-zero terms in (\ref{eqn:covector}).
If we choose $\lambda_{i j} = \delta_{i \sigma (j)}$ for some $\sigma$, then $v^* (v_1, \ldots, v_k) = f_\sigma$.
Therefore $\| v^* \|_C \ge \max_\sigma | f_\sigma |$, and in particular, $\| v^* \|_2 \le \sqrt{N} \, \| v^* \|_C \le \sqrt{d} \, \| v^* \|_C$.
\end{proof}

We point out the following immediate consequence of the preceding proof.

\begin{lem} \label{lem:comass-est}
For every $k$-covector $v^*$ and every orthonormal basis $B$ of $T_x M$, there exist vectors $v_1, \ldots, v_k \in B$ such that $\| v^* \|_C \ge v^* (\pm v_1, v_2, \ldots, v_k) \ge {m \choose k}^{- 1 / 2} \| v^* \|_2$.
\end{lem}

\begin{rmk}
The inequalities in Lemma~\ref{lem:norms-equiv} are not necessarily sharp for all pairs of positive integers $m$ and $k$.
Below we find optimal constants in the two cases $k = 1$ or $2$ of interest in this paper. \qed
\end{rmk}

\begin{lem}
$\| \cdot \|_C = \| \cdot \|_2$ for any $k$-covector and any $k$-form if $k = 1$ or $m - 1$.
\end{lem}

\begin{proof}
For $k = 1$ the proof is an immediate consequence of Lemma~\ref{lem:norms-equiv} since any $1$-covector is simple.
We also give a direct argument.
It again suffices to prove the lemma for covectors.

If $k = 1$, write $v^* = \gamma (v)$ for a (unique) vector $v \in T_x M$.
Then by definition, $\| v^* \|_C \ge \gamma (v) (v / \| v \|_2) = \| v \|_2 = \| v^* \|_2$.
Conversely, for $u \in T_x M$ a unit vector, $| \gamma (v) (u) | = | g (v, u) | \le \| v \|_2$ by the Cauchy-Schwarz inequality, so $\| v^* \|_C \le \| v \|_2$.

For $k = m - 1$ the proof is similar, once we observe that $( \det (M_{\sigma_1}), \ldots, \det (M_{\sigma_m}) )$ is the cross product of $v_1, \ldots, v_{m - 1}$, and $v^* (v_1, \ldots, v_{m - 1})$ is the determinant of the matrix with columns the vectors $f_{v^*}, v_1, \ldots, v_{m - 1}$.
(Geometrically, the latter is, up to sign, the volume of the parallelepiped spanned by these vectors.)
\end{proof}

A key ingredient in our argument in section~\ref{sec:lin-eps-non-squeez} is the following lemma.
We state and prove it in this section for its corollary.

\begin{lem} \label{lem:standard-form}
Let $\omega$ be a two-form on an inner product space $V$.
Then there exists an orthonormal basis $B$ for $V$, $S = \{ u_1, \ldots, u_n, v_1, \ldots, v_n \} \subset B$, $2 n \le \dim V$, and positive numbers $0 < \lambda_1 \le \ldots \le \lambda_n$, such that $\omega (u_j, v_k) = \lambda_j^2 \delta_{j k}$ and $\omega (u_j, u_k) = \omega (v_j, v_k) = 0$ for $1 \le j, k \le n$, and $\omega$ vanishes on $B \, \backslash \, S$.
In other words, $\omega$ can be written in the form $\omega = \sum_{j = 1}^n \lambda_j^2 \, \alpha_j \wedge \beta_j$ with one-forms $\alpha_j$ and $\beta_j$ dual to the elements of $S$.
Moreover, $\omega$ is non-degenerate if and only if $2 n = \dim V$.
\end{lem}

\begin{cor}
$\| \cdot \|_C \le \| \cdot \|_2 \le (\left\lfloor \frac{m}{2} \right\rfloor)^{1 / 2} \| \cdot \|_C$ for any two-covector and two-form on an $m$-dimensional Riemannian manifold $M$, and these inequalities are sharp.
\end{cor}

\begin{proof}[Proof of Corollary]
See the last three sentences of the proof of Lemma~\ref{lem:norms-equiv}.
To verify that the  second inequality is also sharp, suppose that $\omega$ is non-degenerate, and that $J$ is an almost complex structure that is compatible with $g$ so that $\omega = g (J \cdot, \cdot)$.
Then $\| \omega \|_C = \| g (J \cdot, \cdot) \|_C = 1$.
On the other hand, $\| \omega_0 \|_2 = \sqrt{n}$ for the standard symplectic structure $\omega_0$ and standard Riemannian metric on $\R^{2 n}$.
\end{proof}

\begin{proof} [Proof of Lemma~\ref{lem:standard-form}]
The argument here is taken from \cite[Section~1.7.3]{federer:gmt69}.
Let $A$ be the skew-symmetric matrix so that $\omega (v, w) = g (A v, w)$.
Decompose $V$ into a direct sum of mutually orthogonal and $A$-invariant subspaces $W_1, \ldots, W_s$ with $\dim W_j \le 2$ (which exists since $V$ has a basis of eigenvectors of the  symmetric matrix $A^2$), and observe that $\omega (v, w) = 0$ whenever $v \in W_j$ and $w \in W_k$ with $j \not= k$.
Choose an orthonormal basis $u_j$, $v_j$ for each $W_j$ that is two-dimensional, and extend to any orthonormal basis for $V$ if rank$(A) = 2 n < \dim V$.
Note that $g (A u_i, u_i) = \omega (u_i, u_i) = 0$, so we may choose $v_i$ parallel to $A u_i$, and then $\omega (u_i, v_i) = \| A u_i \|$.
Reorder the $W_j$ if necessary.
\end{proof}

\begin{rmk}
Alternatively, one may argue as in \cite[Lemma~2.4.5]{mcduff:ist17} in case $\omega$ is non-degenerate.
Here one observes that the matrix $i A \colon \C^{2 n} \to \C^{2 n}$ is Hermitian, and obtains the same basis vectors $u_j$ and $v_j$ (up to rescaling) as real and imaginary parts of the eigenvectors corresponding to the eigenvalues $i \lambda_j^2$ of $A$.
Note that the signs in the proof of \cite[Lemma~2.4.5]{mcduff:ist17} are different because $A$ is defined there with the opposite sign choice compared to the proof given above.
The argument easily extends to degenerate $\omega$. \qed
\end{rmk}

For $v$ a vector, denote by $\iota_v$ the interior multiplication (or contraction) of a co-vector $v^*$ (of degree $k \ge 1$) by $v$, i.e.\ $\iota_v v^* = v^* (v, \cdot, \ldots, \cdot)$, and similarly, for a vector field $X$, write $\iota_X$ for interior multiplication of a differential form by $X$.

\begin{lem} \label{lem:int-mult-est}
$\| \iota_v v^* \|_C \le \| v \|_2 \| v^* \|_C$ and $\| \iota_v v^* \|_2 \le \sqrt{k} \, \| v \|_2 \| v^* \|_2$ for a $k$-covector $v^*$, and thus $\| \iota_X \beta \|_C \le \| X \|_2 \| \beta \|_C$ and $\| \iota_X \beta \|_2 \le \sqrt{k} \, \| X \|_2 \| \beta \|_2$ for a $k$-form $\beta$.
\end{lem}

\begin{proof}
It is again enough to prove the lemma for covectors.
For the norm $\| \cdot \|_C$, the lemma follows immediately from the definition by writing $\iota_v v^* = \| v \|_2 \cdot \iota_{(v / \| v \|_2)} v^*$.
For the norm $\| \cdot \|_2$, the claim follows from the identity $\| v^* \|_2 = \| f_{v^*} \|_2$ established in the course of the proof of Lemma~\ref{lem:norms-equiv} and the Cauchy-Schwarz inequality.
\end{proof}

\begin{rmk} \label{rmk:int-mult-est}
The inequalities in the previous lemma are sharp for the comass norm, but not for the norm $\| \cdot \|_2$ when $1 < k < \dim M$. \qed
\end{rmk}

\section{Epsilon-area-preserving in dimension two} \label{sec:surfaces}
This short section gives a proof of the main theorem in the dimension two case.

\begin{pro}
Let $M_1$ and $M_2$ be surfaces equipped with area forms $\omega_1$ and $\omega_2$, respectively, $g$ be a Riemannian metric on $M_1$, and $\epsilon \ge 0$.
If $\varphi_k \colon M_1 \to M_2$ is a sequence of $\epsilon$-symplectic embeddings that converges uniformly (on compact subsets) to another embedding $\varphi \colon M_1 \to M_2$, then the limit $\varphi$ is again $\epsilon$-symplectic.
\end{pro}

Therefore in dimension two, Theorem~\ref{thm:eps-symp-rig} holds for any $\epsilon \ge 0$ and with $E = \epsilon$.

\begin{proof}
Let $f_k$ and $f \colon M_1 \to \R$ be the (nowhere vanishing) functions defined by $\varphi_k^* \omega_2 = f_k \omega_1$ and $\varphi^* \omega_2 = f \omega_1$.
By considering the measures obtained by integrating these area forms, we see that the hypothesis $\varphi_k \to \varphi$ uniformly on compact subsets implies that $f_k \to f$ uniformly on compact subsets.
Then
\[ \| \varphi^* \omega_2 - \omega_1 \|_2 = | f - 1 | \| \omega_1 \|_2 = \lim_{k \to \infty} | f_k - 1 | \| \omega_1 \|_2 = \lim_{k \to \infty} \| \varphi_k^* \omega_2 - \omega_1 \|_2 \le \epsilon. \]
See the first paragraph of the proof of Theorem~\ref{thm:eps-symp-rig} below for an interpretation of the final inequality in the case of non-compact manifolds.
\end{proof}

\section{Epsilon-symplectic embeddings} \label{sec:eps-symp}
Let $(M_1, \omega_1)$ and $(M_2, \omega_2)$ be symplectic manifolds of the same dimension, $g$ be a Riemannian metric on $M_1$, and $\epsilon \ge 0$.
Recall that an embedding $\varphi \colon M_1 \to M_2$ is called $\epsilon$-symplectic if $\| \varphi^* \omega_2 - \omega_1 \|_2 \le \epsilon$.

\begin{rmk} \label{rmk:inverse-of-eps-symp}
Let $g_2$ be a Riemannian metric on $M_2$, and $\varphi \colon M_1 \to M_2$ be an $\epsilon_1$-symplectic diffeomorphism.
Let $\beta = \varphi^* \omega_2 - \omega_1$.
Then $(\varphi^{- 1})^* \omega_1 - \omega_2 = - (\varphi^{- 1})^* \beta$.
Thus if $\epsilon_1 > 0$ and $n > 1$, the inverse diffeomorphism $\varphi^{- 1}$ is not necessarily $\epsilon_2$-symplectic for some $\epsilon_2 \ge 0$ that depends only on $\epsilon_1$ and the metric $g_2$. 
If $\Phi$ is a linear map, a similar remark holds for its transpose $\Phi^T$.
See the following example. \qed
\end{rmk}

\begin{exa} \label{exa:inv-trans-eps-symp}
Let $n \ge 2$, and $\omega_0 = \sum_{j = 1}^n dx_j \wedge dy_j$ be the standard symplectic structure on $\R^{2 n} = \R^4 \times \R^{2 n - 4}$.
Let $K \not= 0$.
Consider the isomorphism $\Psi$ of $\R^4$ defined by $\Psi (x_1, y_1, x_2, y_2) = (x_1, y_1 + \epsilon x_2, - K^{- 1} x_2, - K y_2)$, and let $\Phi = \Psi \times \id$.
Then $\Phi^* \omega_0 - \omega_0 = \epsilon \, dx_1 \wedge dx_2$.
On the other hand, $(\Phi^{- 1})^* \omega_0 - \omega_0 = (K \epsilon) \, dx_1 \wedge dx_2$ and $(\Phi^T)^* \omega_0 - \omega_0 = (- K \epsilon) \, dy_1 \wedge dy_2$. \qed
\end{exa}

The preceding remark and example mean that our proof of Theorem~\ref{thm:eps-non-squeez-expand} cannot follow too closely the standard proof in the symplectic case given in \cite[Section~2.4]{mcduff:ist17}.
We include the following result solely for the sake of completeness.

\begin{lem} \label{lem:ball-stretch}
Let $\varphi \colon M_1 \to M_2$ be an embedding, and suppose that $\psi_1$ and $\psi_2$ are symplectic diffeomorphisms of $M_1$ and $M_2$, respectively.
Then there exists a constant $C (\psi_1)$ so that $\| (\psi_2 \circ \varphi \circ \psi_1)^* \omega_2 - \omega_1 \| \le C (\psi_1) \| \varphi^* \omega_2 - \omega_1 \|$ for both the norm $\| \cdot \|_2$ and the comass norm $\| \cdot \|_C$.
In fact, we may choose $C (\psi_1) =\| (\psi_1)_* \|^2$, where $\| \psi_* \| = \sup_{x \in M} \| d\psi (x) \|$, and $\| \Psi \| = \max \{ \| \Psi v \|_2 \mid \| v \|_2 = 1 \}$.
\end{lem}

\begin{proof}
The lemma follows from the identity $(\psi_2 \circ \varphi \circ \psi_1)^* \omega_2 - \omega_1 = \psi_1^* (\varphi^* \omega_2 - \omega_1)$.
See \cite[Section~1.7.6]{federer:gmt69} for the estimate $\| \psi_1^* (\varphi^* \omega_2 - \omega_1) \|_2 \le C (\psi_1) \| \varphi^* \omega_2 - \omega_1 \|_2$.
(It is sufficient to prove this for two-covectors or a dual inequality for two-vectors.)
The analogous inequality is obvious for the comass norm.
\end{proof}

We will use the following obvious remark in our argument in section~\ref{sec:lin-eps-non-squeez}.

\begin{rmk} \label{rmk:eps-symp-comp-symp}
If $\psi$ is an (anti-)symplectic diffeomorphism of $(M_2, \omega_2)$, then an embedding $\varphi \colon M_1 \to M_2$ is $\epsilon$-(anti-)symplectic if and only if the composition $\psi \circ \varphi$ is $\epsilon$-symplectic. \qed
\end{rmk}

\section{Epsilon-symplectic embeddings into Euclidean space} \label{sec:eps-symp-r2n}
In this section, we consider the case $M = \R^m$ with its standard Riemannian metric $g_0 = \langle \cdot, \cdot \rangle$.
If $m = 2 n$, we also equip $\R^{2 n}$ with its standard symplectic structure $\omega_0 = \sum_{j = 1}^n dx_j \wedge dy_j$.
Recall that $\omega_0 = g_0 (J_0 \cdot, \cdot)$, where $J_0$ is the standard (almost) complex structure on $\R^{2 n}$.

Let $U \subset \R^m$ be an open subset that is star-shaped with respect to the origin.
The following lemma is an immediate consequence of (the proof of) the Poincar\'e Lemma \cite[4.18]{warner:fdm83}.
Let $\Omega^k = \Omega^k (U) = \Omega^k (U, \R)$ be the space of (differential) $k$-forms on $U$, and write as usual $d \colon \Omega^{k - 1} \to \Omega^k$ for the differential.
The cases of greatest interest to us are the open ball $B_r^m \subset \R^m$ of radius $r > 0$ (centered at the origin) and ellipsoids (centered at the origin), but we also have in mind the open polydisk $P (r_1, \ldots, r_n) = B_{r_1}^2 \times \cdots \times B_{r_n}^2 \subset \R^{2 n}$.

\begin{lem}[Quantitative Poincar\'e Lemma] \label{lem:quant-PL}
For each $k \ge 1$, there is a bounded and $\R$-linear (and hence continuous) transformation $h_k \colon \Omega^k \to \Omega^{k - 1}$ such that $h_{k + 1} \circ d + d \circ h_k = \id$, and in particular, the restriction of $h_k$ to the space of closed $k$-forms is a right inverse to the differential $d$.
In fact,
\[ \| h_k (\beta) (x) \|_2 \le \frac{\| x \|_2}{k - 1} \, \sqrt{k {m \choose k - 1}} \max_{0 \le t \le 1} \| \beta (t x) \|_2 \le \frac{s}{k - 1} \, \sqrt{k {m \choose k - 1}} \| \beta \|_2 \]
for $k > 1$, and $\| h_k (\beta) (x) \|_2 \le \| x \|_2 \sqrt{m} \max_{0 \le t \le 1} \| \beta (t x) \|_2 \le s \sqrt{m} \| \beta \|_2$ if $k = 1$, where $s = \sup_{x \in U} \| x \|_2$.
\end{lem}

\begin{proof}
Let $h_k = \alpha_{k - 1} \circ \iota_X = \iota_X \circ \alpha_k$, where $X$ denotes the radial vector field $\sum_{i = 1}^m x_i \cdot (\partial / \partial x_i)$, and where for each $k \ge 0$, $\alpha_k$ is defined by
\[ \alpha_k \left( f_\sigma (x) \, dx_{\sigma (1)} \wedge \ldots \wedge dx_{\sigma (k)} \right) = \left( \int_0^1 t^{k - 1} f_\sigma (t x) \, dt \right) dx_{\sigma (1)} \wedge \ldots \wedge dx_{\sigma (k)}, \]
and then extended linearly to all of $\Omega^k$.
More explicitly, for each $k \ge 1$,
\begin{align*}
& h_k \left( f_\sigma (x) \, dx_{\sigma (1)} \wedge \ldots \wedge dx_{\sigma (k)} \right) = \\
& \int_0^1 t^{k - 1} f_\sigma (t x) \, dt \cdot \sum_{j = 1}^k (- 1)^{j + 1} x_{\sigma (j)} dx_{\sigma (1)} \wedge \ldots \wedge dx_{\sigma (j - 1)} \wedge dx_{\sigma (j + 1)} \wedge \ldots \wedge dx_{\sigma (k)}.
\end{align*}
This definition is given in \cite[4.18]{warner:fdm83} only for the unit ball $B_1^m$.
However, the definition is the same for $B_r^m$ and arbitrary star-shaped $U$.
In fact, if $D_r \colon \R^m \to \R^m$ denotes the dilation $x \mapsto r \cdot x$, then we can define $h_k^r = (D_r^{- 1})^* \circ h_k \circ D_r^*$ on $\Omega_k (B_r^m)$, and the definitions of $\alpha_k$, $\iota_X$, and $h_k$ are invariant under conjugation by the induced isomorphism $D_r^*$.
See \cite[4.18]{warner:fdm83} for a proof that $h_{k + 1} \circ d + d \circ h_k = \id$.

To verify the claimed estimates, first note that by Lemma~\ref{lem:int-mult-est},
\[ \| (\iota_X \beta) (x) \|_2 \le \sqrt{k} \, \| x \|_2 \, \| \beta (x) \|_2 \le \sqrt{k} \, s \, \| \beta \|_2. \]
Moreover, with the notation from section~\ref{sec:norms},
\begin{align}
\| \alpha_k (\beta) (x) \|_2 & = \sqrt{\sum_\sigma \left( \int_0^1 t^{k - 1} f_\sigma (t x) dt \right)^2} \nonumber \\
& \le \frac{1}{k} \sqrt{\sum_\sigma \max_{0 \le t \le 1} f_\sigma^2 (t x)} \label{eqn:first-estimate} \\
& \le \frac{1}{k} \sqrt{m \choose k} \max_\sigma \max_{0 \le t \le 1} | f_\sigma (t x) | \label{eqn:second-estimate} \\
& \le \frac{1}{k} \sqrt{m \choose k} \max_{0 \le t \le 1} \| \beta (t x) \|_2, \label{eqn:third-estimate}
\end{align}
which yields the inequality for $h_k = \alpha_{k - 1} \circ \iota_X$ if $k > 1$ and for $h_1 = \iota_X \circ \alpha_1$.
\end{proof}

\begin{cor} \label{cor:quant-PL}
If all of the coefficients $f_\sigma$ of a differential $k$-form $\beta$, $k \ge 1$, are constant along rays through the origin, then $\| h_k (\beta) (x) \|_2 \le \frac{\| x \|_2}{\sqrt{k}} \| \beta (x) \|_2 \le \frac{s}{\sqrt{k}} \| \beta \|_2$.
\end{cor}

\begin{proof}
Apply the argument in the previous proof to $h_k = \iota_X \circ \alpha_k$ and observe that $\| \alpha_k (\beta) (x) \|_2 = \frac{1}{k} \| \beta (x) \|_2$.
\end{proof}

\begin{rmk}
Note that the following estimates in the course of the proof of Lemma~\ref{lem:quant-PL} are sharp: (\ref{eqn:first-estimate}) for example when all $f_\sigma$ are constant and equal, (\ref{eqn:second-estimate}) when all $f_\sigma$ are equal and nonzero, and (\ref{eqn:third-estimate}) when at most one $f_\sigma$ is nonzero.
The combination of the inequalities however may not be sharp (as in Corollary~\ref{cor:quant-PL}). \qed
\end{rmk}

\begin{rmk} \label{rmk:hodge-decomp}
On a general closed and oriented Riemannian manifold, one has the Hodge decomposition $d \circ \delta \circ G + \delta \circ G \circ d + H = \id$, where $\delta = \pm \star d \star$, $G$ is Green's operator, and $H$ is the projection to harmonic forms \cite[Chapter~6]{warner:fdm83}.
An explicit estimate as in Lemma~\ref{lem:quant-PL} for the norm of $(\delta \circ G) (\beta) = (G \circ \delta) (\beta)$ in terms of the norm of an exact form $\beta$ is much more challenging. \qed
\end{rmk}

\begin{rmk}
In general, a family of linear transformations $\sigma_k \colon \Omega^k \to \Omega^{k - 1}$ such that $d \circ \sigma_k = \id$ for all $k \ge 1$ is called a splitting (of the de Rham complex).
In that language, the maps $h_k$ in Lemma~\ref{lem:quant-PL} and $\delta \circ G$ in Remark~\ref{rmk:hodge-decomp} are splittings.
These splittings are smooth in the sense that the constructions depend smoothly on the differential form.
A splitting also exists in the case of non-compact manifolds (that are countable at infinity), see \cite[Section~1.5]{banyaga:scd97} for a summary. \qed
\end{rmk}

\begin{rmk}
Even more generally, given two maps $\varphi$ and $\psi$ between smooth manifolds, a family of linear transformations $h_k \colon \Omega^k \to \Omega^{k - 1}$ such that $\varphi^* - \psi^* = h_{k + 1} \circ d + d \circ h_k$ is called a homotopy operator between $\varphi$ and $\psi$.
In that language, the linear transformations in Lemma~\ref{lem:quant-PL} define a homotopy operator between the identity map and a constant map \cite[Remark~4.19]{warner:fdm83}.

A.~Banyaga \cite[Section~3.1]{banyaga:scd97} constructed a homotopy operator $I_{ \{ \varphi_t \} }$ between a (compactly supported) diffeomorphism that is isotopic to the identity and the identity:
let $\{ \varphi_t \}_{0 \le t \le 1}$ be an isotopy with $\varphi_0 = \id$ and $\varphi_1 = \varphi$, and $X = \{ X_t \}_{0 \le t \le 1}$ the unique vector field that generates this isotopy.
Then $I_{\{ \varphi_t \}} (\beta) = \int_0^1 \varphi_t^* (\iota_{X_t} \beta) dt$.
In particular, if $\beta$ is a closed $k$-form, then $\varphi^* \beta - \beta = d I_{\{ \varphi_t \}} (\beta)$.
However, the $(k - 1)$-form $I_{\{ \varphi_t \}} (\beta)$ is not necessarily small when $\beta$ is small, unless the isotopy $\{ \varphi_t \}$ is already known to be $C^1$-small (or $k = 1$ and the $C^0$-norm $\| X \|$ is small). \qed
\end{rmk}

\begin{lem} \label{lem:eps-symp-approx-symp}
Suppose that $\varphi \colon B_r^{2 n} \to \R^{2 n}$ is an $\epsilon$-symplectic embedding such that $0 \le \epsilon < 1 / \sqrt{2}$, and define $\rho > 0$ by
\[ \rho = (1 - \sqrt{2} \, \epsilon)^{\sqrt{2 n}} \le 1. \]
Then there exists an embedding $\psi \colon B_{\rho r}^{2 n} \to B_r^{2 n}$ such that $B_{\rho s}^{2 n} \subset \psi (B_s^{2 n}) \subset B_{\rho^{- 1} s}^{2 n}$ for all $s \le \rho r$, $\| \psi (x) - x \|_2 \le (\rho^{- 1} - 1) \| x \|_2$ for all $x \in B_{\rho r}^{2 n}$, and $(\varphi \circ \psi)^* \omega_0 = \omega_0$.
If $\varphi$ is (the restriction of) a linear map $\R^{2 n} \to \R^{2 n}$, then we may choose
\[ \rho = \sqrt{1 - \sqrt{2} \epsilon} \le 1. \]
\end{lem}

\begin{proof}[Proof]
We construct the embedding $\psi$ as the time-one map of an isotopy $\psi_t$ defined on the smaller ball $B_{\rho r}^{2 n}$ and with $\psi_0 = \id$ using Moser's argument, and so that $(\varphi \circ \psi)^* \omega_0 = \psi^* (\varphi^* \omega_0) = \omega_0$.

Let $\omega_t = \omega_0 + t (\varphi^* \omega_0 - \omega_0)$.
Since $\epsilon < 1$, each two-form $\omega_t$ is symplectic, and $\frac{d}{dt} \omega_t = \varphi^* \omega_0 - \omega_0$ is a closed (and hence exact) two-form.
Let $h_2$ be the linear transformation in Lemma~\ref{lem:quant-PL}, and consider the one-form $\sigma = h_2 (\varphi^* \omega_0 - \omega_0)$.
By Lemma~\ref{lem:quant-PL} and by hypothesis, $\| \sigma (x) \|_2 \le 2 \sqrt{n} \, \| x \|_2 \| \varphi^* \omega_0 - \omega_0 \|_2 \le 2 \sqrt{n} \, \epsilon \| x \|_2$.

Following Moser's idea, define a family of vector fields $X_t$ by $\iota_{X_t} \omega_t = - \sigma$.
Then
\begin{align*}
\| (\iota_{X_t} \omega_t) (x) \|_2 & \ge \| (\iota_{X_t} \omega_0) (x) \|_2 - t \, \| (\iota_{X_t} (\varphi^* \omega_0 - \omega_0)) (x) \|_2 \\
& \ge \| (\iota_{X_t} \omega_0) (x) \|_2 - \sqrt{2} \, t \, \| X_t (x) \|_2 \| \varphi^* \omega_0 - \omega_0 \|_2 \\
& = \| X_t (x) \|_2 - \sqrt{2} \, t \, \| X_t (x) \|_2 \| \varphi^* \omega_0 - \omega_0 \|_2 \\
& \ge (1 - \sqrt{2} \, \epsilon \, t) \| X_t (x) \|_2.
\end{align*}
We used Lemma~\ref{lem:int-mult-est} for the second inequality.
Thus $\| X_t (x) \|_2 \le C (t) \| x \|_2$, where
\[ C (t) = \frac{\sqrt{2 n} \cdot \sqrt{2} \, \epsilon}{1 - \sqrt{2} \, \epsilon \, t}. \]
The trajectories $x \colon [0, 1] \to \R^{2 n}$ of $X_t$ therefore satisfy (away from the fixed point at the origin) the differential inequality
\[ \left| \frac{d}{dt} \| x (t) \|_2 \right| = \frac{\left| 2 \langle x (t), x' (t) \rangle \right|}{2 \| x (t) \|_2} \le \| x' (t) \|_2 = \| X_t (x (t)) \|_2 \le C (t) \| x (t) \|_2, \]
and thus $\| x (0) \|_2 \, (1 - \sqrt{2} \, \epsilon \, t)^{\sqrt{2 n}} \le \| x (t) \|_2 \le \| x (0) \|_2 \, (1 - \sqrt{2} \, \epsilon \, t)^{- \sqrt{2 n}}$.
Finally,
\[ \| x (1) - x (0) \| \le \int_0^1 \| x' (t) \| dt \le \int_0^1 \frac{C (t) \| x (0) \|}{(1 - \sqrt{2} \epsilon t)^{\sqrt{2 n}}} dt = \| x (0) \| ((1 - \sqrt{2} \epsilon)^{- \sqrt{2 n}} - 1), \]
where we dropped the subscript from the norm $\| \cdot \|_2$ for better readability.
Thus the map $\psi$ has all of the stated properties.
For the statement about linear maps, substitute the estimate in Corollary~\ref{cor:quant-PL} into the above argument.
\end{proof}

\begin{exa} \label{exa:rescale-factors}
Suppose $\varphi^* \omega_0 = \sum_{j = 1}^n c_j^2 \, dx_j \wedge dy_j$, $c_j > 0$.
Then the construction in the previous lemma yields $\psi (r_1, \theta_1, \ldots, r_n, \theta_n) = (c_1 r_1, \theta_1, \ldots, c_n r_n, \theta_n)$, where $x_j = r_j \cos \theta_j$ and $y_j = r_j \sin \theta_j$ are polar coordinates on each $\R^2$-factor. \qed
\end{exa}

\begin{rmk} \label{rmk:arbitrary-metric}
Suppose that $g = \langle \cdot, A (x) \cdot \rangle$ is a Riemannian metric on $\R^m$, where $A (x)$ is a non-singular symmetric matrix that depends smoothly on $x \in \R^m$.
Then
\[ \| A^{- 1} (x) \|^{- k / 2} \| v^* \|_{g_0} \le \| v^* \|_g \le \| A (x) \|^{k / 2} \| v^* \|_{g_0} \]
for any $k$-covector $v^* \in \Lambda^k (T_x \R^m)$, where $\| A (x) \| = \sup \{ \| A (x) v \|_2 \mid \| v \|_ 2 = 1 \}$, and
\[ \| A^{- 1} \|^{- k / 2} \| \beta \|_{g_0} \le \| \beta \|_g \le \| A \|^{k / 2} \| \beta \|_{g_0} \]
for any $k$-form $\beta$, where $\| A \| = \sup_{x \in U} \| A (x) \|$.
In particular, all estimates with respect to the standard metric $g_0$ hold for an arbitrary Riemannian metric $g$ up to a constant factor that depends on the metric $g$ only. \qed
\end{rmk}

\section{Linear epsilon-non-squeezing and non-expanding} \label{sec:lin-eps-non-squeez}
We will show in this section that linear $\epsilon$-symplectic maps are characterized by the property that they preserve the linear symplectic width of ellipsoids up to an error that depends continuously on $\epsilon$ and converges to zero as $\epsilon \to 0^+$.
The key observation is that the failure to be symplectic can be expressed quantitatively in terms of the symplectic spectrum of ellipsoids (centered at the origin).

We identify $\R^{2 n}$ with $\C^n$ in the usual way with $z = (x, y)$ corresponding to $x + i y$ for $x, y \in \R^n$.
Recall that with this identification, $\Sp (2 n) \cap \O (2 n) = \U (n)$, where $\Sp (2 n)$,  $\O (2 n)$, and $\U (n)$ denote the groups of symplectic, orthogonal, and unitary matrices, respectively.
We do not distinguish between a matrix and the linear map $\R^{2 n} \to \R^{2 n}$ or $\C^n \to \C^n$ it represents.

\begin{rmk}
For a singular matrix $\Phi$ we have $\| \Phi^* \omega_0 - \omega_0 \|_2 \ge 1$ (cf.\ the proof of Lemma~\ref{lem:eps-symp-approx-symp}).
Thus an $\epsilon$-symplectic matrix with $\epsilon < 1$ is always non-singular. \qed
\end{rmk}

For a non-singular matrix $A$, denote the ellipsoid $A B_1^{2 n}$ (the image of the closed unit ball) by $E (A) = \{ z \in \R^{2 n} \mid \langle z, ((A^{- 1})^T A^{- 1}) \, z \rangle \le 1 \}$.
For $0 < r_1 \le \cdots \le r_n$, consider the diagonal matrix $\Delta (r_1, \ldots, r_n) \colon \C^n \to \C^n$ whose diagonal entries are $r_1, \ldots, r_n$ (in that order), and abbreviate $E (r_1, \ldots, r_n) = E (\Delta (r_1, \ldots, r_n))$, i.e.
\[ E (r_1, \ldots, r_n) = \left\{ z \in \C^n \mid \sum_{j = 1}^n \left| \frac{z_j}{r_j} \right|^2 \le 1 \right\}. \]
Recall that for each ellipsoid $E (A)$, there exists a symplectic matrix $\Psi$ such that $\Psi E (A) = E (r_1, \ldots, r_n)$ for some $n$-tuple $(r_1, \ldots, r_n)$ with $0 < r_1 \le \cdots \le r_n$, and that is uniquely determined by $A$.
It is called the symplectic spectrum of $E (A)$, and the number $r_1$ is its linear symplectic width \cite[Section~2.4]{mcduff:ist17}.
In fact, $r_j^2 = \alpha_j$, where $\pm i \alpha_j$ are the (purely imaginary) eigenvalues (counted with multiplicities) of the matrix $A^T J_0 A$ \cite[Lemma~2.4.6]{mcduff:ist17}.
Recall in this context that $A$ is symplectic if and only if $A^T J_0 A = J_0$, and the latter has eigenvalues $\pm i$ (with multiplicity $n$).
We will generalize the following lemma to a quantitative result for $\epsilon$-symplectic matrices.

\begin{thm} \label{thm:symp-spectrum}
A linear map $\Phi \colon \R^{2 n} \to \R^{2 n}$ is symplectic or anti-symplectic if and only if it preserves the symplectic spectrum of ellipsoids (centered at the origin).
\end{thm}

\begin{proof}
By \cite[Theorem~2.4.4]{mcduff:ist17}, a linear map is symplectic or anti-symplectic if and only if it preserves the linear symplectic width of ellipsoids (centered at the origin).
It only remains to show that a linear map that preserves the linear symplectic width also preserves the entire symplectic spectrum of ellipsoids (centered at the origin).
This statement is implicitly contained in (the proof of) \cite[Lemma~2.4.6]{mcduff:ist17}; we will give a short direct argument here.

Assume that $\Phi$ preserves the linear symplectic width of ellipsoids centered at the origin.
Let $E$ be such an ellipsoid.
We may assume that $E = E (r_1, \ldots, r_n)$.
Write $(R_1, \ldots, R_n)$ for the symplectic spectrum of the image ellipsoid $\Phi E$.
Denote by $A = \Delta (1, \ldots, 1, a, 1, \ldots, 1) \colon \C^n \to \C^n$ the diagonal matrix with $a > 0$ in the $j$-th position and all other diagonal entries equal to $1$.
Then for $a$ sufficiently small, the linear symplectic width of the ellipsoid $A E (r_1, \ldots, r_n) = E (A \Delta (r_1, \ldots, r_n))$ is $a r_j$.
Since $A$ is a diagonal matrix, it commutes with $\Phi$, so that $\Phi E (A \Delta (r_1, \ldots, r_n)) = A \Phi E$.
The latter has linear symplectic width $a R_j$ for $a$ sufficiently small, and thus the assumption on the linear map $\Phi$ implies that $r_j = R_j$.
\end{proof}

We will later make use of the following lemma.

\begin{lem}
If $A$ is a non-singular matrix, $E (A)$ has linear symplectic width $r_1$, and $a > 0$, then the ellipsoid $E (a A)$ has linear symplectic width $a r_1$.
\end{lem}

\begin{proof}
The linear map $x \mapsto a x$ is conformally symplectic.
\end{proof}

The following definition is motivated by Lemma~\ref{lem:eps-symp-approx-symp} and Example~\ref{exa:rescale-factors}.

\begin{dfn}[Epsilon-non-squeezing and non-expanding] \label{dfn:eps-non-squeeze-expand}
Let $0 \le \epsilon < 1$, and $0 < \rho = \sqrt{1 - \epsilon} \le 1$.
For a non-singular matrix $A$, let $r_A > 0$ be the smallest number such that the ellipsoid $E (A) \subset B_{r_A}^{2 n}$, and let $s_A = (1 + \| A^{- 1} \| (\rho^{- 1} - 1) r_A)^{- 1} \le 1$.
If $\| A^{- 1} \| (\rho^{- 1} - 1) r_A < 1$, define $e_A = (1 - \| A^{- 1} \| (\rho^{- 1} - 1) r_A)^{- 1} \ge 1$.

(a) A linear map $\Phi$ has the linear $\epsilon$-non-squeezing property if for each ellipsoid $E (A)$ with linear symplectic width $r_1$ such that the image ellipsoid $\Phi E (A)$ has linear symplectic width $R_1$, the inequality $s_A \, r_1 \le R_1$ holds.

(b) The linear map $\Phi$ has the linear $\epsilon$-non-expanding property if (for $r > 0$) the linear symplectic width of the ellipsoid $\Phi B_r^{2 n}$ is at most $\rho^{- 1} r$, and moreover, for each ellipsoid $E (A)$ with linear symplectic width $r_1$ and $\| A^{- 1} \| (\rho^{- 1} - 1) r_A < 1$, the linear symplectic width $R_1$ of the image ellipsoid $\Phi E (A)$ satisfies the inequality $R_1 \le e_A \, r_1$.
\end{dfn}

\begin{proof}[Proof of Proposition~\ref{pro:eps-non-squeez-expand}]
Let $\rho = \sqrt{1 - \epsilon'}$, and $\psi \colon \R^{2 n} \to \R^{2 n}$ be the embedding from Lemma~\ref{lem:eps-symp-approx-symp}, so that the composition $\Phi \circ \psi$ is symplectic.
Let $E (A)$ be as in Definition~\ref{dfn:eps-non-squeeze-expand}(a) with $\epsilon$ replaced by $\epsilon'$ everywhere.
Then $x \in E (s_A A)$ implies
\[ \| A^{- 1} (\psi (x)) \|_2 \le \| A^{- 1} x \|_2 + \| A^{- 1} \| \| \psi (x) - x \|_2 \le s_A + \| A^{- 1} \| (\rho^{- 1} - 1) s_A r_A \le 1, \]
i.e.\ $\psi (E (s_A A)) \subset E (A)$, and in particular, $(\Phi \circ \psi) E (s_A A) \subset \Phi E (A)$.
Note that $\psi$ need not be linear in general.
However, the restriction of any (relative) symplectic capacity to ellipsoids (centered at the origin) equals (up to a factor $\pi$) the square of the linear symplectic width \cite[Example~12.1.7]{mcduff:ist17} (see also section~\ref{sec:rigidity}), and hence the above inclusion implies $s_A \, r_1 \le R_1$.
Thus $\Phi$ is $\epsilon'$-non-squeezing.

In the two situations of Definition~\ref{dfn:eps-non-squeeze-expand}(b), $\psi (B_{\rho^{- 1} r}^{2 n}) \supset B_r^{2 n}$, and $x \in \partial E (e_A A)$ implies $\| A^{- 1} (\psi (x)) \|_2 \ge 1$, respectively.
The argument for $\epsilon'$-non-expanding is then analogous to the above argument for $\epsilon'$-non-squeezing.
\end{proof}

\begin{rmk} \label{rmk:non-squeez-fail}
There exists no non-squeezing result in any form for $\epsilon \ge 1$.
Indeed, for any $\delta > 0$, the linear map $(x_1, y_1, \ldots, x_n, y_n) \mapsto (\delta x_1, \delta y_1, x_2, y_2, \ldots, x_n, y_n)$ is $1$-symplectic and maps the unit ball to a (symplectic) cylinder of radius $\delta$.
More generally, let $\omega$ be a symplectic form and $\delta$ be the non-degeneracy radius around $\omega$, i.e.\ the supremum over all numbers $d$ so that $\| \omega' - \omega \|_2 \le d$ implies that $\omega'$ is non-degenerate.
Again use Moser's argument for any $d < \delta$.
Thus $\epsilon$-symplectic does not guarantee any form of non-squeezing beyond (and at) the threshold $\epsilon = \delta$. \qed
\end{rmk}

We will use the following obvious remark in the proof of Theorem~\ref{thm:eps-symp-rig} below.

\begin{rmk} \label{rmk:eps-non-squeezing-comp-symp}
Let $\Phi$ be a linear map, and $\Psi$ be a symplectic or anti-symplectic linear map.
Then $\Phi$ has the linear $\epsilon$-non-squeezing ($\epsilon$-non-expanding) property if and only if $\Psi \circ \Phi$ does.
Compare to Remark~\ref{rmk:eps-symp-comp-symp}. \qed
\end{rmk}

\begin{lem} \label{lem:singular-matrix}
A bounded subset $U \subset \R^{2 n}$ that is contained in a hyperplane $H$ is contained in an ellipsoid of arbitrarily small linear symplectic width.
In particular, a singular matrix $\Phi$ does not have the linear $\epsilon$-non-squeezing property for any $\epsilon$.
\end{lem}

\begin{proof}
Let $u$ be a vector that is orthogonal to $H$, and $v$ be a vector that belongs to $H$ so that $\omega_0 (u, v) > 0$.
After rescaling $v$ if necessary, we may assume that the absolute value of the scalar projection to $v$ of any vector in $U$ is bounded by $1$.
Let $R > 0$.
After rescaling $u$ if necessary, we may assume that $\omega_0 (u, v) = R^2$.
Choose a symplectic basis $\{ u_1, \ldots, u_n, v_1, \ldots, v_n \}$ of $\R^{2 n}$ so that $u_1 = R^{- 1} u$ and $v_1 = R^{- 1} v$, and a symplectic matrix $\Psi$ that maps this basis to the standard basis of $\R^{2 n}$.
Then $\Psi U \subset Z_R^{2 n} = B_R^2 \times \R^{2 n - 2}$.
That proves the first claim.
The second claim follows by considering $U = \Phi B_1^{2 n}$ and any (positive) radius $R < \rho$.
\end{proof}

\begin{proof}[Proof of Theorem~\ref{thm:eps-non-squeez-expand}]
The proof is given in six steps.

\emph{Step 1.} Apply Lemma~\ref{lem:standard-form} to $\omega = \Phi^* \omega_0$ (with $A = \Phi^T J_0 \Phi$) to find an orthonormal basis $B = \{ u_1, \ldots, u_n, v_1, \ldots, v_n \}$ of $\R^{2 n}$ and numbers $0 \le \lambda_1 \le \cdots \le \lambda_n$ so that $(\Phi^* \omega_0) (u_j, v_k) = \delta_{j k} \lambda_j^2$ and $(\Phi^* \omega_0) (u_j, u_k) = (\Phi^* \omega_0) (v_j, v_k) = 0$ for $1 \le j, k \le n$.
By Lemma~\ref{lem:singular-matrix}, we may assume that $\lambda_1 > 0$, and therefore $\Phi B$ is (up to rescaling) a symplectic basis of $\R^{2 n}$.
By composing $\Phi$ (on the left) with a symplectic matrix, we may assume that $\Phi u_j = \lambda_j e_j$ and $\Phi v_j = \lambda_j f_j$, where $\{ e_1, \ldots, e_n, f_1, \ldots, f_n \}$ denotes the standard symplectic basis of $\R^{2 n}$.
In particular, $\Phi$ maps the unit ball $B_1^{2 n}$ to the standard ellipsoid $E (\lambda_1, \ldots, \lambda_n)$.
The hypotheses imply that $\rho \le \lambda_1 \le \rho^{- 1}$.

\emph{Step 2.}
For $j = 1, \ldots, n$, write $\mu_j = \sqrt{| \omega_0 (u_j, v_j) |} \le 1$.
Fix an index $j$ with $1 \le j \le n$, and abbreviate $u = u_j$, $v = v_j$, $\lambda = \lambda_j$, and $\mu = \mu_j$.
Let $0 < a \le 1$.
For the remainder of this proof, let $A$ denote the linear map defined by $A u = a u$, $A v = a v$, and $A$ is the identity on $S = \textrm{span} (B \backslash \{ u, v \})$.
Write $r_1$ for the linear symplectic width of the ellipsoid $E (A)$.
The volume of $E (A)$ yields the constraints $a \le r_1 \le \sqrt[n]{a}$.
We will improve these estimates to symplectic estimates and in terms of the number $0 \le \mu \le 1$ as follows.

There exists a unit vector $w \in S$ and $0 \le s, t \le 1$ with $s^2 + t^2 = 1$ so that the vectors $u$ and $J_0 u = s v + t w$ span a unitary disk $D \subset B_1^{2 n}$ of radius $1$, and in particular, $| \omega_0 (u, s v + t w) | = 1$.
By the Cauchy-Schwarz inequality, the latter implies that $s = | \omega_0 (u, v) | = \mu$ and $t = | \omega_0 (u, w) | = \sqrt{1 - \mu^2}$.
Therefore
\begin{align} \label{eqn:lin-symp-width-estimate}
a^2 \le r_1^2 = | \omega_0 (a u, a \mu v + \sqrt{1 - \mu^2} \, w) | = a^2 \mu^2 + a (1 - \mu^2) \le a.
\end{align}
On the other hand, the linear symplectic width $R_1$ of the ellipsoid $\Phi E (A)$ is the smaller of the two numbers $\lambda_1$ and $a \lambda$.

\emph{Step 3.}
By the $\epsilon$-non-expanding hypothesis, $\min (\lambda_1, a \lambda) \le e_A r_1$ for all numbers $a$ with $\rho^{- 1} - 1 < a \le 1$ (so that the number $e_A$ is well-defined).
We will show that for appropriate choices of $a$, $\rho > e_A \sqrt{a}$, and thus by step 1 and by (\ref{eqn:lin-symp-width-estimate}), $\lambda_1 \ge \rho > e_A \sqrt{a} \ge e_A r_1$.
Then $a \lambda \le e_A r_1 \le e_A a$.
Therefore $\lambda_j \le e_A < \rho \, a^{- 1 / 2}$ for all $j = 1, \ldots, n$ and for all numbers $a$ as above.

Consider the function $f (a) = a^{3 / 2} - \rho a + 1 - \rho$, $\rho^{- 1} - 1 < a \le 1$.
Then the condition $\rho > e_A \sqrt{a}$ translates to the inequality $f (a) < 0$.
The (absolute) minimum of the function $f (a)$ is achieved at the point $a = \frac{4}{9} \rho^2$.
Therefore the inequality $f (a) < 0$ has a solution if and only if $f (\frac{4}{9} \rho^2) = - \frac{4}{27} \rho^3 + 1 - \rho < 0$.
The cubic equation $z^3 + \frac{27}{4} z = \frac{27}{4}$ has a single real root $z_0 = \frac{3}{2} ( (1 + \sqrt{2})^{1 / 3} + (1 - \sqrt{2}))^{1 / 3} )$.
Thus for $\epsilon < 1 - z_0^2$, the inequality $\rho > e_A \sqrt{a}$ can be solved for some $a \ge \frac{4}{9} \rho^2$.
Note that $\frac{4}{9} \rho^2 > \rho^{- 1} - 1$ is equivalent to $\frac{4}{9} \rho^3 + \rho - 1 > 0$, and the latter is greater or equal to $- f (\frac{4}{9} \rho^2)$ and thus indeed is positive.

The equation $f (a) = 0$ can be solved in closed form by making the substitution $a = (c \rho)^2$, which leads to the cubic equation $c^3 - c^2 + \rho^{- 3} (1 - \rho) = 0$.
Since $f (\frac{4}{9} \rho^2) < 0$, $f (\rho^2) = 1 - \rho \ge 0$, and $f' (a) > 0$ for $a > \frac{4}{9} \rho^2$, there exists a single real root $c_\rho$ with $2 / 3 < c_\rho \le 1$.
By the last sentence of the first paragraph of this step, $\lambda_j \le c_\rho^{- 1}$ for all $j = 1, \ldots, n$.
For a more explicit estimate in terms of $\rho$, observe that $f (\rho^6) = (\rho^7 (\rho + 1) - 1) (\rho - 1)$, and $\rho^7 (\rho + 1) - 1 \ge 2 \rho^8 - 1 > 0$, provided that $\rho > (\frac{1}{2})^{1 / 8}$.
Therefore if $\epsilon < 1 - (\frac{1}{2})^{1 / 4}$, then $\lambda_j \le \rho^{- 2}$ for all $j = 1, \ldots, n$.
In particular, $c_\rho \to 1^-$ as $\rho \to 1^-$.

\emph{Step 4.}
As in step~2, fix an index $j$, and drop the subscripts from the notation.
We will prove in this step that $\mu \to 1^-$ as $\rho \to 1^-$.

Define the function $g (a) = (\mu^2 - \lambda^2) a^2 + (1 - \mu^2 - 2 \lambda^2 (\rho^{- 1} - 1)) a - (\lambda (\rho^{- 1} - 1))^2$.
Again by (\ref{eqn:lin-symp-width-estimate}), the $\epsilon$-non-squeezing hypothesis $s_A r_1 \le \min (\lambda_1, a \lambda) \le a \lambda$ for all $a > 0$ guarantees that $g (a) \le 0$ for all $0 < a \le 1$.
If $\mu \ge \lambda$, then $\mu \ge \rho \to 1^-$ as $\rho \to 1^-$, and there is nothing more to prove.
Thus we assume henceforth that $\mu < \lambda$.

The (absolute) maximum of $g (a)$ is achieved at the point
\[ a_0 = \frac{1 - \mu^2 - 2 \lambda^2 (\rho^{- 1} - 1)}{2 (\lambda^2 - \mu^2)}. \]
We derive at a contradiction if the maximum $g (a_0)$ is positive and $0 < a_0 \le 1$.
We distinguish three cases:

Case (i).
$a_0 \le 0$.
Since $\mu < \lambda$, this is equivalent to $1 - \mu^2 - 2 \lambda^2 (\rho^{- 1} - 1) \le 0$.
Then $\mu \ge \sqrt{1 - 2 \lambda^2 (\rho^{- 1} - 1)} \to 1^-$ as $\rho \to 1^-$ since $\lambda$ is bounded by step~3.

Case (ii).
$a_0 > 1$, or equivalently, $\mu > \sqrt{2 \lambda^2 \rho^{- 1} - 1}$.
Since $\lambda \ge \lambda_1 \ge \rho$, the latter is bounded from below by $\sqrt{2 \rho - 1} \to 1^-$ as $\rho \to 1^-$.

Case (iii).
$0 < a_0 \le 1$ and $g (a_0) \le 0$.
The latter is equivalent to the inequality
\[ \mu^2 - 2 \lambda (\rho^{- 1} - 1) \sqrt{\lambda^2 - \mu^2} \ge 1 - 2 \lambda^2 (\rho^{- 1} - 1), \]
and in particular, $\mu \ge \sqrt{1 - 2 \lambda^2 (\rho^{- 1} - 1)} \to 1^-$ as $\rho \to 1^-$ as in case (i).

In particular, from cases (i) to (iii) we deduce that $\mu > 0$; since $c_\rho > 2 / 3$, it suffices that the condition $\epsilon < 1 - z_0^2$ from step 3 implies that $\rho \ge 9 / 11$.

\emph{Step 5.}
In this step, we use a standard non-squeezing argument to show that the numbers $\omega_0 (u_j, v_j) = \pm \mu_j^2$ all have the same sign for $1 \le j \le n$.

By composing $\Phi$ (on the left) with a diagonal matrix $\Psi$ with entries equal to $\pm 1$, we may assume that for each $j$ the pairs of numbers $\omega_0 (u_j, v_j)$ and $\omega_0 (\Phi u_j, \Phi v_j)$ have the same sign, and by rearranging the basis $B$ from step 1 if necessary, that these numbers are all positive.
We will see in the next step that then $\Psi \circ \Phi$ is $\epsilon'$-symplectic for some $\epsilon' \ge 0$ with $\epsilon' \to 0^+$ as $\epsilon \to 0^+$.
In particular, for $\epsilon' < 1 / \sqrt{2}$, $\Psi \circ \Phi$ is $\epsilon''$-non-squeezing for some $\epsilon'' \ge 0$ by Proposition~\ref{pro:eps-non-squeez-expand}.
But by a standard squeezing argument (see the proof of \cite[Theorem~2.4.2]{mcduff:ist17}), $\Psi$ squeezes the unit ball $B_1^{2 n}$ into a symplectic cylinder of arbitrarily small radius unless its diagonal entries all have the same sign.
That shows that $\Psi$ must be symplectic or anti-symplectic.

\emph{Step 6.}
By post-composing with the anti-symplectic matrix $\Psi$ from the previous step if necessary, we may assume that $\omega_0 (u_j, v_j) = \mu_j^2$ for all $1 \le j \le n$.
Recall that $\| \omega_0 \|_2 = \sqrt{n}$.
Thus
\begin{align*}
\| \Phi^* \omega_0 - \omega_0 \|_2^2 = \sum_{j = 1}^n (\lambda_j^2 - \mu_j^2)^2 + (n - \sum_{j = 1}^n \mu_j^4) \to 0^+
\end{align*}
as $\rho \to 1^-$, or equivalently, as $\epsilon \to 0^+$.
That proves the theorem.
\end{proof}

\begin{rmk}
Let the matrix $A$ and basis $B$ be as in the preceding proof, and also write $B$ for the matrix with columns the vectors in $B$.
Since $B$ is orthogonal, it is symplectic if and only if it is also complex, i.e.\ commutes with $J_0$.
The deviation of $A$ from being conformally symplectic is measured by the symplectic spectrum of $E (A)$, see (\ref{eqn:lin-symp-width-estimate}).
A measure of the failure of $\Phi$ to be symplectic is therefore the collection of numbers $\lambda_j$ (conformality) and $\pm \mu_j$ (failure to commute with $J_0$). \qed
\end{rmk}

\section{Capacities and rigidity} \label{sec:rigidity}
In this section we prove Theorem~\ref{thm:eps-symp-rig}.
The proof follows closely the argument in the symplectic case ($\epsilon = 0$) given in \cite[Section~12.2]{mcduff:ist17}.

Recall (from \cite[Section~12.1]{mcduff:ist17}) that a (normalized symplectic) capacity on $\R^{2 n}$ is a functor $c$ that assign to an (arbitrary) subset $U \subset \R^{2 n}$ a non-negative (possibly infinite) number $c (U)$ such that the following axioms hold:
\begin{itemize}
\item (monotonicity) if there exists a symplectic embedding $\psi \colon U \to \R^{2 n}$ such that $\psi (U) \subset V$, then $c (U) \le c (V)$,
\item (conformality) $c (a U) = a^2 c (U)$, and
\item (normalization) $ c (B_1^{2 n}) = \pi = c (Z_1^{2 n})$.
\end{itemize}
Moreover, the restriction of any capacity $c$ to ellipsoids (centered at the origin) equals $c (E) = \pi r_1^2$, where $r_1$ denotes as before the linear symplectic width of $E$ \cite[Example~12.1.7]{mcduff:ist17}.
By the monotonicity and conformality axioms, if $a \ge 1$ and $U \subset \R^{2 n}$ are such that $a^{- 1} E \subset U \subset a E$, then $a^{- 2} c (E) \le c (U) \le a^2 c (E)$.
More generally, the restriction of a capacity to compact convex sets is continuous with respect to the Hausdorff metric \cite[Exercise~12.1.8]{mcduff:ist17}.
Therefore, an $\epsilon$-symplectic embedding preserves the capacity of ellipsoids up to an error that converges to zero as $\epsilon \to 0^+$, see Proposition~\ref{pro:eps-pres-cap}.
Note that translations in $\R^{2 n}$ are symplectic and thus preserve capacity, so we can consider ellipsoids with arbitrary center.

\begin{dfn}
Let $U \subset \R^{2 n}$ and $\varphi \colon U \to \R^{2 n}$ be an embedding.
Let $\epsilon \ge 0$, and $s_A \le 1$ and $e_A \ge 1$ be as in Definition~\ref{dfn:eps-non-squeeze-expand}.
Then $\varphi$ is said to preserve the capacity of ellipsoids up to $\epsilon$ if $s_A^2 c (E) \le c (\varphi (E)) \le e_A^2 c (E)$ for every ellipsoid $E = E (A) \subset U$ (where the second inequality holds whenever the number $e_A$ is defined). \qed
\end{dfn}

\begin{pro} \label{pro:eps-pres-cap}
Let $0 \le \epsilon \le 1 / \sqrt{2}$, and $\epsilon'$ be as in Proposition~\ref{pro:eps-non-squeez-expand}.
Then an $\epsilon$-symplectic embedding preserves the capacity of ellipsoids up to $\epsilon'$.
\end{pro}

\begin{proof}
Proposition~\ref{pro:eps-non-squeez-expand}.
\end{proof}

\begin{pro}
Let $\varphi_k \colon B_r^{2 n} \to \R^{2 n}$ be embeddings that preserve the capacity of ellipsoids up to $\epsilon$, and converge uniformly (on compact subsets) to an embedding $\varphi \colon B_r^{2 n} \to \R^{2 n}$.
Then the limit $\varphi$ again preserves the capacity of ellipsoids up to $\epsilon$.
\end{pro}

\begin{proof}
By definition and the above continuity properties of capacities.
\end{proof}

\begin{rmk}
It is actually not necessary to assume that the maps $\varphi_k$ in the preceding proposition are embeddings.
See \cite[Lemma~12.2.3]{mcduff:ist17}. \qed
\end{rmk}

\begin{pro} \label{pro:pres-cap-eps-symp-anti-symp}
Suppose that an embedding $\varphi \colon B_r^{2 n} \to \R^{2 n}$ preserves the capacity of ellipsoids up to $\epsilon$.
Then for $\epsilon \ge 0$ sufficiently small, $\varphi$ is $\epsilon'$-symplectic or $\epsilon'$-anti-symplectic, where $\epsilon' = K (\epsilon)$ is as in Theorem~\ref{thm:eps-non-squeez-expand} (and converges to zero as $\epsilon \to 0^+$).
\end{pro}

\begin{proof}
Let $x \in B_r^{2 n}$.
By composing with translations (which are symplectic), we may assume that $\varphi (x) = x = 0$ is the origin in $\R^{2 n}$.
Recall that $D_h \colon \R^{2 n} \to \R^{2 n}$ denotes rescaling by the factor $h \not= 0$.
Note that for $| h | < 1$, the numbers $s_A$ and $e_A$ in Definition~\ref{dfn:eps-non-squeeze-expand} are rescaling invariant, i.e.\ $s_{h A} = s_A$ and $e_{h A} = e_A$.
Thus $(D_h^{- 1} \circ \varphi \circ D_h)$ preserves the capacity of ellipsoids up to $\epsilon$.
By Proposition~\ref{pro:eps-pres-cap}, the derivative $d \varphi (0) = \lim_{h \to 0} (D_h^{- 1} \circ \varphi \circ D_h)$ also preserves the capacity of ellipsoids up to $\epsilon$.
Then by Theorem~\ref{thm:eps-non-squeez-expand}, $d \varphi (0)$ is either $\epsilon'$-symplectic or $\epsilon'$-anti-symplectic.
For $\epsilon' < 1$, continuity of $d \varphi (x)$ implies that the latter is either $\epsilon'$-symplectic for all $x$ or $\epsilon'$-anti-symplectic for all $x$.
\end{proof}

\begin{proof}[Proof of Theorem~\ref{thm:eps-symp-rig}]
Suppose that $\varphi_k \colon (M_1, \omega_1) \to (M_2, \omega_2)$ is a sequence of $\epsilon$-symplectic embeddings that converges uniformly (on compact subsets) to the limit embedding $\varphi \colon (M_1, \omega_1) \to (M_2, \omega_2)$.
Let $x \in M_1$.
Since $\epsilon$-symplectic is a pointwise condition, we may assume that $(M_1, \omega_1) = (M_2, \omega_2) = (\R^{2 n}, \omega_0)$.
The constants $\delta$ and $E$ depend continuously on $x$, see Remark~\ref{rmk:arbitrary-metric}.
If $M_1$ is compact, let $\delta$ and $E$ be the minimum and maximum over all $x \in M_1$, respectively.
For non-compact manifolds, one needs to replace the constants $\epsilon$, $\delta$, and $E$ by continuous positive functions $\epsilon (x)$, $C (x)$, and $E (x)$ on $M_1$.

By Proposition~\ref{pro:eps-pres-cap}, the limit $\varphi$ preserves the capacity of ellipsoids up to $\epsilon'$.
Then by Proposition~\ref{pro:pres-cap-eps-symp-anti-symp}, $\varphi$ is $E$-symplectic or $E$-anti-symplectic, where $E = K (\epsilon') \to 0^+$ as $\epsilon \to 0^+$.
It only remains to show that the former alternative holds.

There are several ways to argue.
We observe that for $\epsilon$ sufficiently small, $\varphi_k$ is orientation preserving, and thus so is the limit $\varphi$.
If $n$ is odd, and $E$ is sufficiently small, then $\varphi$ cannot be $E$-anti-symplectic.
If $n$ is even, the same argument applies to the map $\varphi \times \id \colon M_1 \times \R^2 \to M_2 \times \R^2$.
Alternatively, let $x \in M_1$, $L$ be a Lagrangian torus in $M_1$ that contains $x$, and $U \subset M_1$ be a tubular neighborhood that can be symplectically identified with a neighborhood of the zero section in $T^* L$ with its standard symplectic structure.
For $k$ sufficiently large, $\varphi_k$ and $\varphi$ are homotopic, and thus induce the same maps on the cohomology groups of $L$.
A symplectic map sends $[\omega_0]$ to itself while an anti-symplectic map reverses sign.
Thus for $\epsilon$ sufficiently small, $\varphi$ must be $E$-symplectic.
\end{proof}

\begin{proof}[Proof of Corollary~\ref{cor:eps-symp-rig}]
Since $E \to 0$ as $\epsilon \to 0$, taking a subsequence of the sequence $\varphi_k$ and applying Theorem~\ref{thm:eps-symp-rig} shows that $\varphi$ is $E$-symplectic for any $E > 0$.
\end{proof}

\section{Shape invariant and epsilon-contact embeddings} \label{sec:shape-contact}
This section outlines an alternative proof of Theorem~\ref{thm:eps-symp-rig} and Corollary~\ref{cor:eps-symp-rig} using the shape invariant, and an adaptation of the proof to $\epsilon$-contact embeddings.
See \cite{eliashberg:nio91, mueller:csc17, sikorav:qpp91, sikorav:rsc89} for details on the shape invariant.

Let $(M, \omega = d\lambda)$ be an exact symplectic manifold of dimension $2 n$, and $T^n$ be an $n$-dimensional torus.
An embedding $\iota \colon T^n \hookrightarrow M$ is called Lagrangian if $\iota^* \omega = 0$; the cohomology class $[\iota^* \lambda] \in H^1(T^n, \R) = \R^n$ is called its $\lambda$-period.

\begin{dfn}[Shape invariant {\cite{eliashberg:nio91}}] \label{dfn:shape}
Let $\tau \colon H^1 (M, \R) \to H^1 (L, \R)$ be a homomorphism.
Then the shape $I (M, \omega ,\tau)$ is the subset of $H^1 (T^n, \R)$ that consists of all points $z \in H^1 (T^n, \R)$ such that there exists a Lagrangian embedding $\iota \colon T^n \hookrightarrow M$ with $\iota^* = \tau$ and $z = [\iota^* \lambda]$, defined up to translation. \qed
\end{dfn}

\begin{thm}[{\cite{eliashberg:nio91, sikorav:rsc89}}] \label{thm:torus-shape}
For $A \subset \R^n$ open and connected, $I (T^n \times A, \lambda_{\textrm can}, \iota_0^*) = A$.
\end{thm}

The role of ellipsoids $E (r_1, \ldots, r_n)$ in this paper can therefore be replaced by products of annuli $A (a_1, b_1, \ldots, a_n, b_n) = (S^1 \times [a_1, b_1]) \times \cdots \times (S^1 \times [a_n, b_n])$, $0 < a_i < b_i$, and the spectrum $(r_1, \ldots, r_n)$ of the ellipsoid $E (r_1, \ldots, r_n)$ by the shape $[a_1, b_1] \times \cdots \times [a_n, b_n]$ of the annulus $A (a_1, b_1, \ldots, a_n, b_n)$.
Similar to the proof given above, $\epsilon$-symplectic embeddings preserve the shape invariant up to an error that converges to zero as $\epsilon \to 0^+$, and this property is preserved by uniform limits (on compact subsets).
See \cite{mueller:csc17} for the $\epsilon = 0$ case.
Details are forthcoming.

We indicate how to prove Corollary~\ref{cor:eps-symp-rig} based on properties of the shape invariant.
An embedding $\iota \colon T^n \hookrightarrow M$ is non-Lagrangian if $\iota^* \omega \not= 0$ (at at least one point), or equivalently, its image is a non-Lagrangian submanifold.

\begin{thm}[\cite{laudenbach:hdl94, mueller:csc17}] \label{thm:rig-lag-hom-class}
Let $\iota \colon T^n \hookrightarrow (\R^{2 n}, \omega_0)$ be a non-Lagrangian embedding.
Then there exists a tubular neighborhood $N$ of $\iota (T^n)$ that admits no Lagrangian embedding $\jmath \colon T^n \hookrightarrow N$ so that the homomorphism $\jmath_* \colon H_1 (T^n, \R) \to H_1 (N, \R)$ is injective.
In particular, the shape $I (N, T^n, \iota^*)$ is empty.
\end{thm}

\begin{proof}[Sketch of proof of Corollary~\ref{cor:eps-symp-rig}]
Let $\iota \colon T^n \hookrightarrow M_1$ be a Lagrangian torus, and $N$ be an arbitrary tubular neighborhood of $(\varphi \circ \iota) (T^n)$.
Let $\psi_k$ be as in Lemma~\ref{lem:eps-symp-approx-symp} (defined on a polydisk $P (r_1, \ldots, r_n)$) so that $\varphi_k \circ \psi_k$ is symplectic.
Then for $k$ sufficiently large, the torus $(\varphi_k \circ \psi_k \circ \iota) (T^n)$ is Lagrangian and contained in $N$, and $\varphi_k \circ \psi_k$ is homotopic to $\varphi$.
By Theorem~\ref{thm:rig-lag-hom-class}, $(\varphi \circ \iota) (T^n)$ must be Lagrangian.
Thus $\varphi$ maps Lagrangian tori to Lagrangian tori, and hence must be conformally symplectic \cite{mueller:csc17}, i.e.\ $\varphi^* \omega_2 = c \, \omega_1$.
That $c = 1$ can be proved using \cite[Proposition~2.29]{mueller:csc17}.
\end{proof}

Let $(M_1, \xi_1)$ and $(M_2, \xi_2)$ be cooriented contact manifolds of the same dimension, and $g_1$ be a Riemannian metric on $M_1$.
An embedding $\varphi \colon M_1 \to M_2$ is called $\epsilon$-contact if there are contact forms $\alpha_1$ on $M_1$ and $\alpha_2$ on $M_2$ so that $\| \varphi^* \alpha_2 - \alpha_1 \| \le \epsilon$ and $\| \varphi^* d\alpha_2 - d\alpha_1 \| \le \epsilon$.
This definition allows the Moser argument in the proof of Lemma~\ref{lem:eps-symp-approx-symp} to go through in the contact setting \cite[page 135f]{mcduff:ist17}, and the proof of $C^0$-rigidity of $\epsilon$-symplectic embeddings in this section can be adapted to $\epsilon$-contact embeddings.
Note that the proof using capacities does not generalize, since the capacity of the symplectization of a contact manifold is infinite.
Compare to \cite{mueller:csc17}.

\section*{Acknowledgments}
We would like to thank Michael Freedman for formulating the question about $\epsilon$-symplectic rigidity answered in this paper, and Yasha Eliashberg for relaying it.

\bibliography{eps-symp-arxiv}
\bibliographystyle{plain}

\end{document}